\documentclass[10pt]{article}

\usepackage[cmtip,all]{xy}
\usepackage[english]{babel}
\usepackage{amsmath}
\usepackage{amsthm}
\usepackage{amsfonts}
\usepackage{verbatim}

\newtheorem{theorem}{Theorem}[section]
\newtheorem{lemma}[theorem]{Lemma}
\newtheorem{proposition}[theorem]{Proposition}
\newtheorem{corollary}[theorem]{Corollary}

\newtheorem{definition}[theorem]{Definition}

\newtheorem{remark}[theorem]{Remark}

\newcommand{\IF}{\mathbb{F}}
\newcommand{\IP}{\mathbb{P}}
\newcommand{\Gal}{\mathcal{G}al}

\begin{document}

\title{Uniqueness of low genus optimal curves over $\IF_{2}$}

\author{Alessandra Rigato}

\date{}
\maketitle

\begin{abstract}
A projective, smooth, absolutely irreducible algebraic curve $X$ of genus $g$ defined over a finite field $\IF_{q}$ is called \emph{optimal} if for every other such genus $g$ curve $Y$ over $\IF_{q}$ one has $\#Y(\IF_{q})\le\#X(\IF_{q})$. In this paper we show that for $g\le 5$ there is a unique optimal genus $g$ curve over $\IF_{2}$. For $g=6$ there are precisely two and for $g=7$ there are at least two.
\end{abstract}

\maketitle

\section{Introduction}

Let $X$ be a projective, smooth and absolutely irreducible genus $g$ curve defined over a finite field $\IF_{q}$. It is well known that the number of $\IF_{q}$-rational points of $X$ is bounded and a lot of research has been done to determine whether the bounds are sharp: see for example Sections 5.2 and 5.3 of \cite{stichtenoth} for an overview. The curve $X$ is called \emph{optimal} if for every other genus $g$ curve $Y$ over $\IF_{q}$ one has $\#Y(\IF_{q})\le\#X(\IF_{q})$. The main result of this paper deals with uniqueness up to $\IF_{2}$-isomorphism of small genus optimal curves defined over $\IF_{2}$. 

\begin{theorem}\label{uniquecurve}
For $g\leq5$, there exists a unique optimal genus $g$ curve defined over $\IF_{2}$. There exist two non-isomorphic genus $6$ optimal curves and at least two non-isomorphic genus $7$ optimal curves defined over $\IF_{2}$.
\end{theorem}

Examples of small genus optimal curves defined over $\IF_{2}$ are already present in \cite{serre}, \cite{serre1} and \cite{niederreiterxing}. In this paper we show that for genus $g\leq 5$ these examples are unique, while one of the genus $6$ curves we construct appears to be new.\\
The proof of this result consists of two steps. We first determine a short list of Zeta functions that an optimal curve over $\IF_{2}$ can have. In Section \ref{sec: zetafunctions} we show that for genus $g\leq 5$ there is only one possible Zeta function, while for $g=6$ there are two. Next we apply class field theory techniques as in \cite{auer}, \cite{lauter}, \cite{schoof}, \cite{serre}, \cite{serre1}, and recent results by Howe and Lauter in \cite{howelauter} to show that for each possible Zeta function there exists precisely one curve. In Section \ref{sec: unique01} we discuss curves of genus $0$ and $1$. Sections \ref{sec: unique2} to \ref{sec: unique6} are devoted to curves of genus $2$ to $6$. Finally, in Section \ref{sec: genus7} we exhibit two optimal genus $7$ curves with different Zeta functions.

\footnotetext{The author wishes to express her gratitude to her advisor Ren\'e Schoof, for this work would not have been possible without his precious help. The author also thanks Everett Howe for his interesting and constructive comments and Claus Fieker for his MAGMA computation. Part of this paper was written while the author was supported by the Fund for Scientific Research Flanders (F.W.O. Vlaanderen)}

\section{Zeta function and real Weil polynomial of a curve}\label{sec: zetafunctions}

Throughout this paper a curve is understood to be projective, smooth and absolutely irreducible over a finite field of definition $\IF_{q}$. In order to study optimal genus $g$ curves defined over $\IF_{q}$ it is of interest to determine the quantity
\[
N_{q}(g):=\max \{\#X(\IF_{q}) \, \arrowvert  \, X \textrm{ is a genus } g \textrm{ curve defined over } \IF_{q} \}.
\]
Then, an optimal genus $g$ curve $X$ defined over $\IF_{q}$ satisfies $\#X(\IF_{q})=N_{q}(g)$. Several methods have been developed in order to determine $N_{q}(g)$ for given $q$ and $g$. The progress is listed and continuously updated in the tables \cite{geervlugt}. In particular Serre determined very good upper bounds for the number of $\IF_{q}$-rational points in \cite{serre1}. For $q=2$ he gives the estimate $\#X(\IF_{2}) \le 0.83 g+5.35$. For $g\ge 2$ this improves the Hasse-Weil bound $\#X(\IF_{q})\leq q+1+\lfloor 2g\sqrt{q} \rfloor$. In \cite{serre} Serre also provided examples of genus $g$ curves defined over $\IF_{2}$ attaining these bounds. Hence for small genus curves he proved that $N_{2}(g)$ is as follows \cite[Theorem 5]{serre1}
\[
\begin{array}{|c|c|c|c|c|c|c|c|c|}
\hline
g & 0 & 1 & 2 & 3 & 4 & 5 & 6 & 7\\
\hline
N_{2}(g) & 3 & 5 & 6 & 7 & 8 & 9 & 10 & 10\\
\hline
\end{array}
\]
The Zeta function of a genus $g$ curve $X$ defined over $\IF_{q}$ is given by
\[
Z(t)=\prod_{d\geq 1}\frac{1}{(1-t^d)^{a_d}},
\]
where 
\[
a_{d}=\#\{ P\, \arrowvert \, P \textrm{ place of } X \textrm{ such that deg}\, P=d\}.
\]
In particular, $a_{1}=\#X(\IF_{q})$. The Zeta function $Z(t)$ is a rational function of the form
\[
Z(t) = \frac{L(t)}{(1-t)(1-qt)},
\]
where 
\[
L(t)  =  \prod_{i=1}^{g}(1-\alpha_{i} t)(1-\overline{\alpha_{i}} t)
\]
for certain $\alpha_{i} \in \mathbb{C}$ of absolute value $\sqrt{q}$. Therefore $L(t)=q^g t^{2g} + b_{2g-1}t^{2g-1}+\ldots + b_{1} t+1 \in \mathbb{Z}[t]$ is determined by the coefficients $b_{1}, \ldots, b_{g}$ which are in turn determined by the numbers $a_{1}, \ldots, a_{g}$. See for example \cite[Section 5.1]{stichtenoth} for more details.\\

To a genus $g$ curve $X$ having $L(t)$ as numerator of its Zeta function, we associate the so-called \emph{real Weil polynomial} of $X$:
\[
h(t)=\prod_{i=1}^{g}(t-\mu_i) \;\; \in\, \mathbb{Z}[t],
\]
where $\mu_{i}=\alpha_{i}+\overline{\alpha_{i}}$ is a real number in the interval $[-2\sqrt{q}, 2\sqrt{q}]$, for all $i=1,\ldots, g$. We have
\begin{equation}\label{relation}
L(t)=t^{g}h(qt+1/t).
\end{equation} 
One can hence turn the problem of determining the Zeta function of $X$ into the problem of determining the real Weil polynomial of $X$. Not every polynomial $h(t)$ with all zeros in the interval $[-2\sqrt{q}, 2\sqrt{q}]$ and with the property that
\[
\frac{L(t)}{(1-t)(1-qt)}=\prod_{d\ge 1}\frac{1}{(1-t^{d})^{a_{d}}}
\]
for certain integers $a_{d}\ge 0$ is necessarily the real Weil polynomial of a curve. The following result is due to Serre \cite[page Se 11]{serre}, \cite[Lemma 1]{lauter}.

\begin{proposition}  \label{Res1}
Let $h(t)$ be the real Weil polynomial of a curve $C$ over $\IF_{q}$. Then $h(t)$ cannot be factored as $h(t) = h_{1}(t)h_{2}(t)$, with $h_{1}(t)$ and $h_{2}(t)$ non-constant polynomials in $\mathbb{Z}[t]$ such that the resultant of $h_{1}(t)$ and $h_{2}(t)$ is $\pm 1$.
\end{proposition}

This result has been generalized by E.~Howe and K.~Lauter. Proposition \ref{Res2} below is an improvement \cite{howelauterslides} of \cite[Theorem 1.b)]{howelauter} and Proposition \ref{EllFact} is \cite[Theorem 1, Proposition 13]{howelauter}. Recall that the \emph{reduced resultant} of two polynomials $f, g \in \mathbb{Z}[t]$ is defined to be the non-negative generator of the ideal $(f,g) \cap \mathbb{Z}$.

\begin{proposition} \label{Res2}
Let $h(t)=h_{1}(t)h_{2}(t)$ be the real Weil polynomial of a curve $C$ over $\IF_{q}$, where $h_{1}(t)$ and $h_{2}(t)$ are coprime non-constant factors in $\mathbb{Z}[t]$. Let $r$ be the reduced resultant of the radical of $h_{1}(t)$ and the radical of $h_{2}(t)$. If $r= 2$, then, there exists a degree $2$ map $C \to C'$, where the curve $C'$ is defined over $\IF_{q}$ and has either $h_{1}(t)$ or $h_{2}(t)$ as real Weil polynomial.  
\end{proposition}

\begin{proposition} \label{EllFact}
Let $h(t) = (t - \mu)h_{2}(t)$ be the real Weil polynomial of a curve $C$ over $\IF_{q}$, where $t - \mu$ is the real Weil polynomial of an elliptic curve $E$ and $h_{2}(t)$ a non-constant polynomial in $\mathbb{Z}[t]$ coprime with $t-\mu$. If $r \neq \pm1$ is the resultant of $t-\mu$ and the radical of $h_{2}(t)$, then $C$ admits a map of degree dividing $r$ to an elliptic curve isogenous to $E$.
\end{proposition}

For a curve $X$ we denote by $a(X)$ the vector $[a_{1}, a_{2},\ldots]$. The main result of this section is the following.

\begin{theorem} \label{RealWeil} 
For $g\le 6$ the real Weil polynomial $h(t)$ and the vector $a(X)$ of an optimal genus $g$ curve $X$ over $\IF_{2}$ are as follows:
\begin{align}
g=1:\; h(t)&=t+2, &a(X)&=[5, 0, 0, 5, 4, 10, \ldots]; \nonumber \\
g=2:\; h(t)&=t^2 + 3t + 1, &a(X)&=[6, 0, 1, 1, 6, 12, \dots]; \nonumber \\ g=3:\; h(t)&=t^3 + 4t^2 + 3t - 1, &a(X)&=[7, 0, 1, 0, 7, 7, \ldots]; \nonumber \\
g=4:\; h(t)&=(t + 1)(t + 2)(t^2 + 2t - 2), &a(X)&=[8, 0, 0, 2, 4, 8, \ldots]; \nonumber \\
g=5:\; h(t)&=t(t + 2)^2(t^2 + 2t - 2), &a(X)&=[9, 0, 0, 2, 0, 12, \ldots]; \nonumber \\
g=6:\; \phantom{h(t)}& & & \nonumber \\
h(t) &= t(t+2)(t^4+5t^3+5t^2-5t-5), &a(X)&=[10,0,0,0,3,10,\dots], \label{rWa} 
\\
h(t) &= (t-1)(t+2)(t^2+3t+1)^2, &a(X)&=[10,0,0,0,2,15,\dots]. \label{rWb}
\end{align}

\end{theorem}

\begin{proof}
Following \cite[page Se Th 38]{serre} we compute for each $g\le 6$ a finite list of monic degree $g$ polynomials $h(t)\in \mathbb{Z}[t]$ for which $a_{1}$ is equal to the number of $\IF_{2}$-rational points of an optimal genus $g$ curve and for which $a_{d}\ge 0$ for $d\ge 2$ in the relation $L(t)=t^{g}h(qt+1/t)$. Moreover we require that $h(t)$ has the property that its zeros are in the interval $[-2\sqrt{2},2\sqrt{2}]$. Finally, we require that the conditions of Proposition \ref{Res1} are satisfied. A short computer calculation gives a unique polynomial for $g\leq 5$ and three polynomials for $g=6$:
\begin{align}
&(1)& h(t)&=t(t+2)(t^4+5t^3+5t^2-5t-5), &a(X)&=[10,0,0,0,3,10,\dots]; \nonumber \\
&(2)& h(t)&=(t-1)(t+2)(t^2+3t+1)^2, &a(X)&=[10,0,0,0,2,15,\dots]; \nonumber \\
&(3)& h(t)&=(t+1)(t+2)(t^2+2t-2)(t^2+2t-1), &a(X)&=[10,0,0,1,0,12,\dots]. \nonumber
\end{align}

We show that the third polynomial cannot occur. The resultant of the factors $t+2$ and $(t+1)(t^2+2t-2)(t^2+2t-1)$ is $-2$. Hence, by Proposition \ref{EllFact}, a genus $g=6$ curve $X$, having this polynomial as real Weil polynomial, admits a degree $2$ map $X \to E$, where $E$ is a genus one curve having real Weil polynomial $t+2$. The curve $E$ has parameters $a(E)=[5, 0, 0, 5, 4, 10, \ldots]$, hence $E$ has five places of degree $4$ while $X$ has only one. Since $E$ does not have any degree $2$ places, this means that one place $Q$ of the degree $4$ places of $E$ must ramify in $X$. The different $D$ of the quadratic function field extension $\IF_{2}(X)/\IF_{2}(E)$ satisfies $2Q \le D$ (where the coefficient $2$ is forced by wild ramification). On the other hand the degree of the different is $2g-2=10=\mathrm{deg}\,D$ by the Hurwitz formula. Thus $D=2Q+2R$, where $R$ is a rational point of $E$. But this is a contradiction because all of five rational points of $E$ split completely in $X$ since $\#X(\IF_{2})=10$.
\end{proof}

\section{Uniqueness of optimal elliptic curves}\label{sec: unique01}

In this section we prove Theorem \ref{uniquecurve} for curves of genus $0$ and $1$.

\begin{remark}
We denote by $\IP^{1}$ the projective line over $\IF_{2}$ and by $0$, $1$ and $\infty$ its three rational points. Over a finite field, every genus $0$ curve is isomorphic to $\IP^{1}$. Therefore $\IP^{1}$ is optimal. The Zeta function of $\IP^{1}$ is
\[
Z(t)=\frac{1}{(1-2t)(1-t)} \qquad \mathrm{and \;hence} \qquad a(\IP^{1})=[3,1,2,3,6,\ldots].
\]
\end{remark}

\begin{proposition}
Up to $\IF_{2}$-isomorphism, the unique genus $1$ curve having five rational points over $\IF_{2}$ is the elliptic curve $E$ of affine equation $y^{2}+y=x^{3}+x$.
\end{proposition}

\begin{proof}
A genus $1$ curve $E$ over $\IF_{2}$ having five rational points over $\IF_{2}$ is an elliptic curve. Hence $E$ admits a separable degree $2$ morphism to $\IP^{1}$. It can be described as a smooth cubic in $\IP^{2}$ of affine equation of the form $y^2+a(x)y=f(x)$, where $a(x)$ and $f(x)$ are polynomials in $\IF_{2}[x]$, the first of degree $0$ or $1$ and the latter of degree $3$ \cite[Appendix A]{silverman}. Since the point at infinity $\infty$ of $\IP^{1}$ ramifies in $E$, one has $a(x)=1$. The affine points $0$ and $1$ of $\IP^{1}$ have to split, thus we have that $f(0)=f(1)=0$ and hence $f(x)=x(x+1)(x+a)$, where $a \in  \IF_{2}$. If $a=1$ we find the equation $y^2+y=x^3+x$ and if $a=0$ the equation $y^2+y=x^3+x^2$. These two curves are indeed isomorphic over $\IF_{2}$ by changing coordinates through the map $(x,y) \mapsto (x+1,y)$.
\end{proof}

\begin{remark}
The function field of the genus $1$ curve $E$ can also be described as the ray class field of $\IP^{1}$ of conductor $4$ times a rational point, in which the other two rational points of $\IP^{1}$ are both split. Since $\mathrm{Aut}(\IP^{1})$ acts doubly transitively on $\{0,1,\infty\}$, different choices give rise to isomorphic ray class fields.
\end{remark}

\begin{remark}\label{remg1N5}
We often refer to this unique optimal elliptic curve $E$ throughout this paper. For future reference, we present here some properties of $E$. In terms of the affine equation $y^2 +y =x^3 +x$, we denote the five rational points of $E$ as follows: we write $P_{0}$ for the point at infinity and we put 
\begin{equation}\label{Epoints}
P_{1}=(0,0), \quad P_{2}=(0,1), \quad P_{3}=(1,0), \quad P_{4}=(1,1).
\end{equation}
The real Weil polynomial of $E$ is $h(t)=t+2$. The vector $a(E)$ of the numbers $a_{d}$ of places of degree $d=1,2,\ldots$ of $E$ is given by
\[
a(E)=[5, 0, 0, 5, 4, 10, 20,\ldots].
\]
Let $a\in \IF_{16}$ be a root of $x^{4}+x+1$. Then, the five places of degree $4$ of $E$ have coordinates 
\begin{align*}
Q_1=(a^{3}&,a^{3}+a),  \, Q_2=(a^{3},a^{3}+a+1), \\
Q_3=(a^{3}+1,a), \, Q_4&=(a^{3}+1, a+1), \, Q_5=(a^{2}+a+1,a). \nonumber
\end{align*}
Let $b\in \IF_{32}$ be a root of $x^{5}+x^{3}+1$, then the four places of degree $5$ of $E$ consist of the points of coordinates:
\begin{equation*}
R_{1}=(b,b^{4}), \, R_{2}=(b,b^{4}+1),\,
R_{3}=(b+1,b^{4}+b), \, R_{4}=(b+1,b^{4}+b+1).
\end{equation*}
Let $c \in \IF_{64}$ be a root of $x^{6}+x^{5}+1=0$. The places of degree $6$ of $E$ have coordinates
\begin{align*}
T_{1}=(c^5+c^3+c^2+c+1,c^5 + c^4 + c^3 + 1 )&, \, T_{2}=(c^5+c^3+c^2+c, c^4+c^2+c), \nonumber \\
T_{3}=(c^3+c^2+1, c^3+c^2+c), \,T_{4}&=(c^3+c^2+1,  c^3+c^2+c+1), \nonumber  \\ 
T_{5}=(c+1, c^4+c^3+c^2+c), \, T_{6}&=(c+1, c^4+c^3+c^2+c+1), \\
T_{7}=(c^3+c^2, c+1 )&,\, T_{8}=(c^3+c^2, c ), \nonumber \\
T_{9}=(c, c^4+c^3+c^2), \, T_{10}&=(c, c^4+c^3+c^2+1).\nonumber 
\end{align*}
The order $5$ automorphism $\sigma$ of $E$ given by addition of $P_{1}$ acts transitively on $E(\IF_{2})$ as follows: $P_{0}\mapsto P_{1}\mapsto P_{3}\mapsto P_{4}\mapsto P_{2}\mapsto P_{0}$. The action of $\sigma$ on the places of degree $4$ is as follows: $Q_{1}\mapsto Q_{5}\mapsto Q_{2}\mapsto Q_{4}\mapsto Q_{3}\mapsto Q_{1}$.
On the other hand, the order $4$ automorphism of $E$
\[
\tau:\, (x,y) \mapsto (x+1,y+x+1)
\]
fixes $P_0$ and acts transitively on the remaining four rational points: $P_{1}\mapsto P_{4}\mapsto P_{2}\mapsto P_{3}\mapsto P_{1}$. Similarly, $\tau$ fixes $Q_{5}$ and acts transitively on the remaining degree $4$ places: $Q_{1}\mapsto Q_{4}\mapsto Q_{2}\mapsto Q_{3}\mapsto Q_{1}$. The action of $\tau$ on the places of degree $5$ is transitive: $R_{1}\mapsto R_{4}\mapsto R_{2}\mapsto R_{3}\mapsto R_{1}$.
\end{remark}

\section{Uniqueness of genus $2$ optimal curves}\label{sec: unique2}

\begin{proof}[Proof of Theorem \ref{uniquecurve} for $g=2$]
A genus $2$ optimal curve $X$ over $\IF_2$ is hyperelliptic. Since $X$ has six rational points, all three rational points of $\IP^1$ split completely in the double covering $X \to \IP^1$. By Theorem \ref{RealWeil}, the curve $X$ has no places of degree $2$ and only one place of degree $3$. Thus only one degree $3$ place $Q$ of the two degree $3$ places of $\IP^1$ totally ramifies in $X$. The different $D$ of the corresponding function field extension is hence $2Q$, since $2Q\leq D$ and $\textrm{deg}\,D=6$ by the Hurwitz formula. Any genus $2$ curve having six rational points over $\IF_2$ is hence a double covering of $\IP^1$ of conductor $2Q$, where $Q$ is a place of $\IP^1$ of degree $3$, in which all rational points of $\IP^1$ are split. A different choice of the degree $3$ place of $\IP^1$ leads to an $\IF_2$-isomorphic curve. Indeed, the $\IF_2$-isomorphism $x \mapsto 1/x$ preserves the rational points of $\IP^{1}$, but switches the two degree $3$ places.
\end{proof}

\section{Uniqueness of genus $3$ optimal curves}

We briefly recall some important results on the Jacobian variety of a curve in order to state and prove a useful lemma.\\

Let $X$ be a curve defined over $\IF_q$. We denote by $\mathcal{J}ac(X)$ the Jacobian variety of $X$ and by $T_\ell$ the Tate module attached to $\mathcal{J}ac(X)$, where $\ell$ is a prime number different from the characteristic of $\IF_q$. We set $V_{\ell}=T_{\ell} \otimes \mathbb{Q}_{\ell}$. Let $F: V_\ell \to V_\ell$ be the Frobenius map and let $V: V_\ell \to V_\ell$ be the Verschiebung map: the unique map such that $V\circ F=q$. Then $\mathbb{Z}[F,V]\subseteq \mathrm{End}(\mathcal{J}ac(X))$. Next we let $\phi$ be the canonical polarization on $\mathcal{J}ac(X)$. Then $\phi$ can be represented as a non-degenerate alternating form $\phi: V_\ell \times V_\ell \to \mathbb{Q}_\ell$. Here $\mathbb{Q}_\ell$ denotes the field of $\ell$-adic numbers. Since $\phi(F(x),F(y)) = q\phi(x,y)$ for every $x, y \in V_\ell$, by bilinearity of $\phi$ we have that $\phi(F(x),F(y)) = q\phi(x,y)=\phi(qx,y)=\phi(V(F(x)),y)$. It follows that $\phi(z,F(y))=\phi(V(z),y)$ for any $y, z \in V_\ell$. In other words $V$ is left adjoint to $F$ \mbox{with respect to $\phi$.}

\begin{theorem} [Torelli Theorem \cite{weil}]
Let $X$ and $X'$ be two curves over a perfect field $k$. Let \mbox{$\tau: \mathcal{J}ac(X) \to  \mathcal{J}ac(X')$} be an isomorphism over $k$ compatible with the canonical polarizations. Then
\begin{enumerate}
\item if $X$ is hyperelliptic, there exists a unique isomorphism$f:\, X \to X'$ over $k$ which gives $\tau$;
\item if $X$ is not hyperelliptic, there exists a unique isomorphism \mbox{$f: \, X \to X'$} over $k$ and a unique $\varepsilon \in \{ \pm 1\} $ such that $f$ gives $\varepsilon \tau$. 
\end{enumerate}
\end{theorem} 

\begin{corollary}
If $\tau$ is an automorphism of $\mathcal{J}ac(X)$ over $k$ preserving the po\-la\-ri\-za\-tion, then either $\tau$ or $-\tau$ comes from an automorphism over $k$ of $X$.
\end{corollary}

\begin{lemma} \label{lemmam7} 
Any genus $3$ curve $X$ having exactly seven rational points over $\IF_{2}$ admits an automorphism of order $7$.
\end{lemma}

\begin{proof}
We show that for a genus $3$ curve $X$ having seven rational points over $\IF_2$ the ring $\mathbb{Z}[F,V]\subseteq \mathrm{End}(\mathcal{J}ac(X))$ is isomorphic to $\mathbb{Z}[\zeta_7]$, the ring of integers of $\mathbb{Q}(\zeta_7)$. The minimal polynomial of $F+V$ is the real Weil polyomial of $X$. By Theorem \ref{RealWeil} this is $h(t)=t^3 + 4t^2 + 3t - 1$. It is an irreducible polynomial of discriminant $7^{2}$. Hence, for a root $\alpha \in \overline{\mathbb{Q}}$ of $h(t)$, the number field $\mathbb{Q}(\alpha)$ is a cyclic extension of degree $3$ of $\mathbb{Q}$, which is ramified only at $7$. By the Kronecker-Weber Theorem the field $\mathbb{Q}(\alpha)$ is hence the unique degree $3$ subfield $\mathbb{Q}(\zeta_{7}+\zeta_{7}^{-1})$ of $\mathbb{Q}(\zeta_{7})$ and $\mathbb{Z}[\alpha]$ is its ring of integers. Consider now the minimal polynomial of Frobenius $x^{2}-\alpha x+2 \in \mathbb{Z}[\alpha][x]$. Its discriminant $\alpha^{2}-8$ has norm $7$ and hence generates a prime ideal $\pi \subseteq \mathbb{Z}[\alpha]$ lying over the prime $7$ of $\mathbb{Z}$. By class field theory $\mathbb{Q}(\alpha)$ admits a unique quadratic extension unramified outside of $\pi$ and the three infinite primes lying over $7$. This is the field $\mathbb{Q}(\zeta_{7})$, which has discriminant $7^{5}$ by the conductor-discriminant formula. The discriminant of $\mathbb{Q}(\alpha,x)$ can be computed to be $7^{5}$ as well by means of the relative discriminant formula for towers of number fields. Hence $\mathbb{Z}[F,V]=\mathbb{Z}[\alpha,x]$ is the ring of integers $\mathbb{Z}[\zeta_{7}]$ of $\mathbb{Q}(\zeta_{7})$ as wanted.\\ 
Now $\mathcal{J}ac(X)$ has in particular an automorphism $\tau$ of order $7$ corresponding to $\zeta_7$. We show that $\tau$ preserves the polarization $\phi$. By bilinearity of $\phi$ and since $V$ is the complex conjugate of $F$, the left adjoint to an element $\tau \in \mathbb{Z}[F,V]$ is its complex conjugate $\overline{\tau}$. Since $\tau$ satisfies $\tau \overline{\tau}=1$, we have in particular that $\phi(\tau (x), y) = \phi(x,\overline{\tau}(y)) = \phi(x, {\tau}^{\scriptscriptstyle{-1}}(y))$ for any $x$, $y$ in $V_\ell$. This implies that $\phi(\tau(x), \tau(y)) = \phi(x, y)$ for any $x, y \in V_\ell$. In other words $\tau$ preserves the polarization $\phi$ of $\mathcal{J}ac(X)$. By the above Corollary of Torelli's Theorem the curve $X$ admits hence an automorphism $f$ of order $7$. Indeed if $f$ does not induce $\tau$ of order $7$, but $f$ induces $-\tau$, then $f^{2}$ is an automorphism of order $7$ of $X$.
\end{proof}

\begin{proof}[Proof of Theorem \ref{uniquecurve} for $g=3$]
By Lemma \ref{lemmam7} the curve $X$ admits an automorphism $f$ of order $7$. Then, by Galois correspondence, $X$ is a cyclic covering of degree $7$ of a curve which can only be $\IP^1$ by comparing the genera and the degree of the different in the Hurwitz formula. By the conductor-discriminant formula, the conductor $D$ of such a covering satisfies $6\,\textrm{deg}D=18$. Since there are seven rational points on $X$, only one rational point $P$ of $\IP^1$ splits completely. Thus one has $D=Q$, where $Q$ is a place of $\IP^1$ of degree $3$. Hence $X$ is a cyclic degree $7$ covering of $\IP^1$ of conductor $Q$, where one rational point $P$ of $\IP^1$ splits completely. Different choices of $P$ in $\{0,1,\infty\}$ and of the degree $3$ place $Q$ give rise to $\IF_2$-isomorphic curves. Indeed, since the automorphisms group of $\IP^1$ acts transitively on the rational points, we can always first reduce to the case $P=\infty$. Next the automorphism $x \mapsto x+1$ fixes $P$ and maps one degree $3$ place of $\IP^{1}$ into the other one.
\end{proof}

\section{Uniqueness of genus $4$ optimal curves}

\begin{proof}[Proof of Theorem \ref{uniquecurve} for $g=4$]
By Theorem \ref{RealWeil} the real Weil polynomial of an optimal genus $4$ curve $X$ over $\IF_2$ is $h(t)=(t + 1)(t + 2)(t^2 + 2t - 2)$. The resultant of the polynomials $t+2$ and $(t + 1)(t^2 + 2t - 2)$ is $2$. Proposition \ref{EllFact} implies therefore that the curve $X$ is a double covering of the unique optimal elliptic curve $E$ having real Weil polynomial $t+2$ described in Remark \ref{remg1N5}. Since $X$ has no places of degree $2$, no rational point of $E$ can be inert in $X$. Hence, since $X$ has eight rational points, there is only one possibility for the five rational points of $E$: three of them split completely and two are totally ramified. Denoting by $P$ and $P'$ the two wildly ramified rational points of $E$, we have that the contribution to the different of the quadratic function field extension $\IF_{2}(X)/\IF_{2}(E)$ is at least $2P+2P'$. Since the degree of the different has to be $6$ by the Hurwitz formula, the different, which is also the conductor of the extension, is $4P+2P'$ or $2P+4P'$. Thus any optimal genus $4$ curve over $\IF_2$ is a double covering of the optimal elliptic curve $E$ of conductor $4P+2P'$ or $2P+4P'$, in which the other three rational points of $E$ split completely. Uniqueness of $X$ follows from the fact that $\mathrm{Aut}(E)$ acts doubly transitively on $E(\IF_{2})$ as described in Remark \ref{remg1N5}.
\end{proof}

\section{Uniqueness of genus $5$ optimal curves}

\begin{lemma}\label{lemmag5}
Let $C$ be the hyperelliptic curve over $\IF_{2}$ of affine equation $y^2 + y = x^5 + x^3$. Let $P$ be a rational point of $C$ and let $K$ be the ray class field of $\IF_{2}(C)$ of conductor $4P$ in which all rational points of $C$ except $P$ split completely. Then $K=\IF_{2}(C)$ except when $P$ is the point at infinity, in which case we have $[K:\IF_{2}(C)]=2$.
\end{lemma}

\begin{proof}
Let $t$ denote a uniformizer at $P$ and let $S=C(\IF_{2})\backslash \{P\}$. By Artin reciprocity the Galois group $\Gal(K/\IF_2(C))$ is isomorphic to the $S$-ray class group of $C$ modulo $4P$ \cite[Section 2.5]{niederreiterxing}. In this case the latter is isomorphic to a quotient of $R=\Big (\IF_2[[t]]/(t^4)\Big)^*\simeq \mathbb{Z}_4 \times \mathbb{Z}_2$ by the $S$-unit group of $C$ \cite[Section 8]{schoof}. We show that if $P$ is the point at infinity of $C$ we have \mbox{$\Gal(K/\IF_2(C))\simeq \mathbb{Z}_2$}. On the other hand, if $P$ is one of the other rational points
\[
P_0=(0,0),\, P_0'=(0,1), \, P_1=(1,0),\textrm{ or } P_1'=(1,1)
\]
of $C$, the group $\Gal(K/\IF_2(C))$ is trivial. A sketch of the computations follows.
\begin{enumerate}
\item[$i)$] Let $P$ be the point at infinity of $C$. Then a basis for the $S$-unit group of $C$ consists of the functions with principal divisors given by
\begin{eqnarray*}
\Big( \frac{y+x^{3}}{x^{3}}\Big) & = &2P_{0}+P_{1}'-3P_{0}',\\
\Big( \frac{y+1}{y}\Big) \;& = &3(P_{0}'-P_{0})+2(P_{1}'-P_{1}),\\
\Big( \frac{x+1}{x}\Big) \;& = &P_{1}-P_{0}+P_{1}'-P_{0}'.
\end{eqnarray*}
Let $t=y/{x^3}$ be a uniformizer at $P$, then their images in $R$ are:
\begin{eqnarray*}
\frac{y+x^{3}}{x^{3}} & \equiv & 1+t\;\,\textrm{mod}\, t^{4},\\
\frac{y+1}{y} \;& = &1+\frac{1}{y} \equiv 1+t^{5}\equiv 1 \;\, \textrm{mod} \, t^{4},\\
\frac{x+1}{x} \;& = &1+\frac{1}{x}\equiv 1+t^{2} \;\, \textrm{mod} \, t^{4},
\end{eqnarray*}
since $1/y=t^{5}+O(t^{6})$ and $1/x=t^{2}+O(t^{4})$. The element $1+t$ generates a subgroup $R'$ of $R$ of index $2$ and $1+t^{2} \in R'$. Therefore $\Gal(K/\IF_{2}(C))\simeq R/R'\simeq \mathbb{Z}_{2}$. 
\item[$ii)$] Let $P=P_0$ and $x$ a uniformizer at $P$. In this case consider the two $\IF_{2}$-linearly independent $S$-units of divisors given by
\begin{eqnarray*}
(x+1) & = &P_{1}+P_{1}'-2P_{\infty},\\
(y+1) & = & 3P_{0}'+2P_{1}'-5P_\infty.
\end{eqnarray*}
Here $P_{\infty}$ denotes the point at infinity of $C$. By means of Hensel's lemma, we compute the local expansion of $y$ at $P_0$ as $y=x^5+x^3+O(x^6)$. Therefore their images in $R$ are
\begin{eqnarray*}
x+1 & \equiv & 1+x \phantom{^{3}}\;\,\textrm{mod} \, x^{4},\\
y+1 & \equiv & 1+x^{3} \;\,\textrm{mod} \, x^{4}.
\end{eqnarray*}
In this case the group $R$ is generated by the images of the $S$-units and thus the quotient group is trivial. 
\end{enumerate}
The other possibilities for $P$ reduce to case $ii)$ by applying the order $4$ automorphism $\varphi:(x,y) \mapsto (x+1, y+x^2+1)$ of $C$. It fixes the point at infinity of $C$ and acts transitively on the other rational points of $C$.
\end{proof}

\begin{proof}[Proof of Theorem \ref{uniquecurve} for $g=5$]
By Theorem \ref{RealWeil} a genus $5$ optimal curve $X$ defined over $\IF_2$ has real Weil polynomial $h(t)=t(t+2)^2(t^2+2t-2)$. Since the principal ideal $(t(t+2),t^2+2t-2) \cap \mathbb{Z}$ is generated by $2$, Proposition \ref{Res2} implies that the curve $X$ is a double covering of a curve $C$ having real Weil polynomial either $t(t+2)^2$ or $t^2+2t-2$. If $C$ had $t(t+2)^2$ as real Weil polynomial, it would be a genus $3$ curve having seven rational points over $\IF_2$, which is impossible by Theorem \ref{RealWeil}.\\
Hence $C$ is a genus $2$ curve having five rational points and no place of degree $2$. Every genus $2$ curve defined over $\IF_2$ is a hyperelliptic curve. Up to $\IF_2$-isomorphism there exists a unique hyperelliptic curve $C$ over $\IF_{2}$ having real Weil polynomial $t^2+2t-2$. Indeed such a hyperelliptic curve has five rational points and no place of degree $2$. Thus the different of the function field extension associated to the double covering $C \to \IP^1$ has to be $6Q$, where $Q$ is a rational point of $\IP^1$. According to the classification of genus $2$ curves over $\IF_{2}$ in \cite[page 327]{maisnernart}, by taking $Q=\infty$, any such hyperelliptic curve is $\IF_2$-isomorphic to a projective curve of affine equation $y^2+y=x^5+ax^3+bx^2+c$, with $a,b,c \in \IF_2$. Of the eight possible equations arising from the choice of the parameters $a,b,c$, only the affine equation $y^2 + y = x^5 + x^3$ describes a projective curve having five rational points over $\IF_{2}$ and no places of degree $2$.\\ 
Since $X$ has nine rational points, only one rational point $P$ of $C$ ramifies in the double covering $X \to C$, while the other four rational points of $C$ split completely in $X$. The different of $\IF_{2}(X)/\IF_{2}(C)$ is hence $4P$, since it must have degree $4$ by the Hurwitz formula. The function field $\IF_{2}(X)$ is hence an abelian extension of $\IF_{2}(C)$ of conductor $4P$, where the other four rational points of $C$ split completely. The maximal among such abelian extensions is the ray class field $K$ described in Lemma \ref{lemmag5}. Hence $P$ is the point at infinity of $C$ and $\IF_{2}(X)=K$.
\end{proof}

\section{Genus $g=6$ optimal curves}\label{sec: unique6}

Theorem \ref{RealWeil} lists the two possible real Weil polynomials of an optimal genus $6$ curve defined over $\IF_2$. In this section we give a proof of the existence of a unique genus $6$ curve for each of the two listed polynomials. 

\begin{proposition}
Up to $\IF_{2}$-isomorphism, there is a unique curve having real Weil polynomial as in \eqref{rWa} of Theorem \ref{RealWeil}.
\end{proposition}

\begin{proof}
Let $X$ be a genus $6$ optimal curve defined over $\IF_2$ having real Weil polynomial $h(t)=t(t+2)(t^4+5t^3+5t^2-5t-5)$. Since the resultant of the factors $t+2$ and $t(t^4+5t^3+5t^2-5t-5)$ is $-2$, there exists a degree $2$ morphism $X \to E$ by Proposition \ref{EllFact}. All of the five rational points of $E$ split completely into the ten rational points of $X$. By the Hurwitz formula the degree of the different of $\IF_{2}(X)/\IF_{2}(E)$ is $10$. Now, since $a_2(X)=a_3(X)=a_4(X)=0$, the different is precisely $2R$, where $R$ is a degree $5$ place of $E$. Thus, any such optimal genus $6$ curve is a double covering of $E$ of conductor $2R$, in which all rational points of $E$ are split. As observed in Remark \ref{remg1N5}, the elliptic curve $E$ has actually four points of degree $5$ and the $\IF_{2}$-automorphism $\tau$ of $E$ acts transitively on them. The choice of a different degree $5$ ramifying point, gives thus an $\IF_2$-isomorphic curve.
\end{proof}

In the rest of the section, let $X$ be a genus $6$ optimal curve over $\IF_2$ having real Weil polynomial as in \eqref{rWb} of Theorem \ref{RealWeil}.

\begin{proposition}\label{Xb}
Up to $\IF_{2}$-isomorphism, there is a unique curve having real Weil polynomial as in \eqref{rWb} of Theorem \ref{RealWeil}.
\end{proposition}

\begin{lemma}\label{nonnormal6}
The curve $X$ is a non-Galois covering of degree $3$ of the elliptic curve $E$ such that $X$ is unramified outside of $E(\IF_{2})$.
\end{lemma}

The following definition introduces a notation for the splitting behavior of the rational points of the elliptic curve $E$.

\begin{definition}\label{abcpointsE}
Let $X\to E$ be a degree $3$ covering defined over $\IF_{2}$. Consider a rational point $P$ of $E$. We say that $P$ is
\begin{enumerate}
\item[$a)$] an $A$-point, if $P$ splits completely in $X$;
\item[$b)$] a $B$-point, if $P$ splits into two points of $X$, one unramified and the other one with ramification index $2$;
\item[$c)$] a $C$-point, if $P$ is totally ramified in $X$ with ramification index $3$.
\end{enumerate}
Moreover we denote by $a$, $b$, $c$ the number of $A$-points, $B$-points and $C$-points of $E$ respectively. 
\end{definition}

\begin{proof}[Proof of Lemma \ref{nonnormal6}]
By Theorem \ref{RealWeil}, the real Weil polynomial of $X$ is $h(t)=(t-1)(t+2)(t^2+3t+1)^2$. Since the resultant of the polynomials $t+2$ and $(t-1)(t^2+3t+1)$ is equal to $3$, by Proposition \ref{EllFact} the curve $X$ admits a morphism of degree $3$ to the optimal elliptic curve $E$ described in Remark \ref{remg1N5}. Since the parameters of $X$ are $a(X)=[10,0,0,0,2,15,\dots]$, there are no places of degree $2$ or $3$ on $X$. Therefore each of the $\IF_2$-rational points in $E$ can hence be either an $A$-point, a $B$-point or a $C$-point in the sense of Definition \ref{abcpointsE}. Then we have
\[
a+b+c  =  5 \quad \textrm{and} \quad 3a+2b+c  = 10,
\]
and hence
\[
2a+b=5 \quad \textrm{and}\quad a=c.
\]
This leaves us with the three cases of Table \ref{typesIII}.
\begin{table}[h]
\begin{center}
\begin{tabular}{| l | c | c | c |}
\hline
 $\,$ & $\mathbf{a}$ & $\mathbf{b}$ & $\mathbf{c}$\\
 \hline
 \textbf{case} $I$ & $0$ & $5$ & $0$\\
 \hline
 \textbf{case} $II$ & $1$ & $3$ & $1$\\
 \hline
 \textbf{case} $III$ & $2$ & $1$ & $2$\\
 \hline
\end{tabular}
\caption{Splitting behavior of the rational points of $E$ in $X$}
\label{typesIII}
\end{center}
\end{table}
In each case the covering $X \to E$ is non-Galois since $b$ is never zero. Moreover the function field extension $\IF_2(X)/\IF_2(E)$ is unramified outside of $E(\IF_2)$. Consider indeed the degree of the different, which is $10$ by the Hurwitz formula. By Definition \ref{abcpointsE}, only one of the two points of $X$ lying over a $B$-point of $E$ is wildly ramified. This gives a contribution to the degree of the different which is at least $2$. The contribution to the different that comes from the rational points of $E$ is therefore at least $db+2c$ with $d\geq 2$. Therefore it is at least $5\cdot 2=10$ in case $I$, at least $3\cdot 2 + 2=8$ in case $II$ and at least $1\cdot 2 + 2\cdot 2=6$ in case $III$. Since there are no points of degree $2$, $3$ or $4$ on $X$, any other non-rational ramified place of $E$ should have degree strictly larger than $4$. But this would give a too large contribution to the different in each of the three cases. Hence there are no other places of $E$ ramifying in $X$ but those of degree one.
\end{proof}

\begin{definition}\label{defioverlineprime}
We denote by $\overline{X}$ the curve whose function field is the normal closure of $\IF_2(X)$ with respect to $\IF_2(E)$: it is a Galois extension of $\IF_{2}(E)$ having Galois group isomorphic to the symmetric group $S_3$. We denote by $X'$ the curve having as function field the quadratic extension of $\IF_2(E)$ corresponding to the group $A_3\simeq\mathbb{Z}_3$, the unique (normal) subgroup of $S_3$ of index $2$. The situation is described in the following picture:
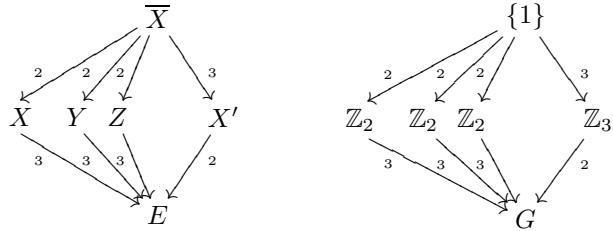
\begin{center}
$
\xymatrix@C=0.1pc { 
&&& \overline{X} \ar @{->}[dlll]_(.68){{\scriptscriptstyle 2}} \ar @{->}[dll]_(.67){{\scriptscriptstyle 2}} \ar @{->}[dl]_(.61){{\scriptscriptstyle 2}}  \ar@{->}[drr]^(.64){{\scriptscriptstyle 3}} && \\
{\qquad X} \ar @{->}[drrr]_(.32){{\scriptscriptstyle 3}} & {\;\;Y} \ar @{->}[drr]_(.33){{\scriptscriptstyle 3}} & **[r]Z \ar @{->}[dr]_(.39){{\scriptscriptstyle 3}} & && X' \ar @{->}[dll]^(.36){{\scriptscriptstyle 2}} \\
&&& E & &} 
\quad$
$
\xymatrix@C=0.1pc { 
&&& \{1\} \ar @{->}[dlll]_(.68){{\scriptscriptstyle 2}} \ar @{->}[dll]_(.67){{\scriptscriptstyle 2}} \ar @{->}[dl]_(.61){{\scriptscriptstyle 2}}  \ar@{->}[drr]^(.64){{\scriptscriptstyle 3}} && \\
{\qquad \mathbb{Z}_2} \ar @{->}[drrr]_(.32){{\scriptscriptstyle 3}} & {\;\;\mathbb{Z}_2} \ar @{->}[drr]_(.33){{\scriptscriptstyle 3}} & \mathbb{Z}_2 \ar @{->}[dr]_(.39){{\scriptscriptstyle 3}} & && \mathbb{Z}_3 \ar @{->}[dll]^(.36){{\scriptscriptstyle 2}} \\
&&& G & &} 
$
\end{center}
\end{definition}

We sum up some arithmetical properties of $X'$ and $\overline{X}$ in the following auxiliary lemmas.

\begin{lemma} \label{abcPoints}
\mbox{}
\begin{enumerate}
\item[$a)$] The $A$-points of $E$ split completely in $\overline{X}$ and $X'$.
\item[$b)$] Over each $B$-point of $E$ there are three points of $\overline{X}$, each with ramification index $2$ and there is one point of $X'$ with ramification index $2$.
\item[$c)$] Over each $C$-point of $E$ there is a unique place of $\overline{X}$ of degree $2$.
\end{enumerate}
\end{lemma}

\begin{proof}
Let $Y$ be the degree $3$ covering of $E$ as in the picture above. 
\begin{enumerate}
\item[$a)$] 
Each $A$-point $P$ of $E$ splits completely over $\IF_{2}(Y)\simeq \IF_{2}(X)$ as well. Hence the function field of $\overline{X}$, being the compositum of $\IF_{2}(X)$ and $\IF_{2}(Y)$, is the splitting field of $P$. Moreover, since the function field of $X'$ is contained in it, $P$ splits completely in $X'$ as well.
\item[$b)$] 
Since there is more than one point of $\overline{X}$ lying over a $B$-point $P$ of $E$, the decomposition groups of the points lying over $P$ have order $2$. Since the ramification index of one of the points of $X$ lying over $P$ is $2$, all points of $\overline{X}$ lying over $P$ have ramification $2$. It also follows that there is a unique point of $X'$ lying over $P$. It has ramification index $2$.
\item[$c)$] 
Let $P$ be a $C$-point of $E$. Since the order of the inertia group of any of the points of $\overline{X}$ lying over $P$ has order divisible by $3$, the same is true for a point $P'$ of $X'$ lying over $P$. It follows that $\overline{X} \to X'$ is a cyclic degree $3$ covering that is ramified at $P'$. Therefore, by class field theory, the multiplicative group of the residue field of $P'$ must have order divisible by $3$. It follows that $P$ must be inert in $X'$. Indeed, in this case the residue field of $P'$ is $\IF_{4}$.\vspace{-0.5cm}
\end{enumerate}
\end{proof}

\begin{lemma}\label{genusx'}
The curve $X'$ is defined over $\IF_2$ and has genus $g'=6-c$. Moreover, the covering $X' \to E$ is ramified exactly at the $B$-points of $E$.
\end{lemma}

\begin{proof}
Since $X \to E$ is unramified outside of $E(\IF_{2})$ by Lemma \ref{nonnormal6}, the same holds for the covers $\overline{X}$ and $X'$ of $E$. Lemma \ref{abcPoints} implies then that $X' \to E$ is ramified precisely at the $B$-points of $E$. By Table \ref{typesIII} there is always at least one such point. Thus, since the residue field of any place contains the constant field, the constant field of $X'$ is $\IF_{2}$.
In order to compute the genus of $X'$ we compare the different $\textrm{Diff}(X'/E)$ of $\IF_2(X')/\IF_2(E)$ with the different $\textrm{Diff}(X/E)$ of $\IF_2(X)/\IF_2(E)$. By the Hurwitz formula we have that $10=2\cdot 6-2=\textrm{deg\,Diff}(X/E)=\textrm{deg}\,\textrm{Diff}(X/E)_{tame} + \textrm{deg}\,\textrm{Diff}(X/E)_{wild}$. The contribution given to $\textrm{Diff}(X/E)$ by the $c$ tamely ramified points is $2c$. Therefore the contribution of the $b$ wildly ramified points is $10-2c$. Since these are precisely the points that are ramified in $X' \to E$, the degree of $\textrm{Diff}(X'/E)$ is also equal to $10-2c$. It follows that $2g'-2=10-2c$, so that $g'=6-c$ as required.
\end{proof}

\begin{lemma}\label{XX'} 
For low degrees $d$, the number $a_{d}$ of places of degree $d$ of the curves $\overline{X}$ and $X'$ are as follows:
\begin{align*}
a_1(\overline{X})&=6a+3b,  &a_1(X')&=2a+b;\\
a_2(\overline{X})&=c ,  &a_2(X')&=c;\\
a_3(\overline{X})&=0 ,  &a_3(X')&=0;\\
a_4(\overline{X})&=0 ,  &a_4(X')&=10;\\
a_5(\overline{X})&=0.
\end{align*}
\end{lemma}

\begin{proof}
The computation of the numbers $a_1(\overline{X})$ and $a_1(X')$ of $\IF_2$-rational points of $\overline{X}$ and $X'$ respectively, follows directly by Lemma \ref{abcPoints}. By the same Lemma the degree $2$ places of $X'$ are precisely the ones lying over the $C$-points of $E$ and they are themselves totally ramified in $\overline{X}$. This gives $a_2(X')=c =a_2(\overline{X})$. Because of Theorem \ref{RealWeil}, the curve $X$ has parameters $a(X)=[10, 0, 0, 0, 2, 15, ...]$ and in particular $a_{3}(X)=0$. Since also $a_{3}(E)=0$, it follows at once that $a_3(X')=a_3(\overline{X})=0$. The curve $X$ has no places of degree $2$ or $4$, thus $a_4(\overline{X})=0$. Moreover this means that the five places of degree $4$ of $E$ are inert in $X$. Since they are not ramified, their decomposition group has to be cyclic and hence of order $3$. Therefore they are split in $X'$ and we have 
\[
a_4(X')=2a_4(E)=2\cdot 5=10. 
\]
Suppose that $a_5(\overline{X})$ is not zero, then one of the places of degree $5$ of $E$ splits completely in $\overline{X}$. This implies that $X$ has at least three places of degree $5$, which is not the case. Therefore $a_5(\overline{X})=0$.
\end{proof}

The following Lemma describes abelian extensions $K_{D}$ of $\IF_{2}(E)$ for particular choices of the conductor $D$. These extensions play a role in the proof of Proposition \ref{possibleX}. The divisor $D$ is a sum of points in $E(\IF_{2})$. See Remark \ref{remg1N5} for the notation. 

\begin{lemma}\label{rayclass2}
Let $K_{D}$ denote the ray class field of $\IF_{2}(E)$ of conductor $D$ in which the point at infinity and all places of degree $4$ of $E$ split completely. Then $K_{D}$ is trivial when $D=4P_{1}+2P_{2}+2P_{3}$ or $D=2P_{1}+2P_{2}+4P_{3}$. It has degree $2$ over $\IF_{2}(E)$ when $D=2P_{1}+4P_{2}+2P_{3}$.
\end{lemma}

\begin{proof}
Let $Q_1,Q_2,\ldots, Q_5$ denote the degree $4$ places of $E$ as listed in Remark \ref{remg1N5} and let $S=\{ Q_1,Q_2,Q_3,Q_4,Q_5,P_0\}$. A basis for the $S$-unit group of $E$ is given by the following functions $u_{i}$,  $i=1,\ldots,5$:
\begin{align*}
u_1  &=  x^4+x^3+x^2+x+1, & \textrm{with } & (u_{1})  =  Q_1+Q_2-8P_0,\\
u_2  &=  x^4+x^3+1,  & \textrm{with } & (u_{2})  =  Q_3+Q_4-8P_0,\\
u_3  &=  x^2+x+1, & \textrm{with } & (u_{3}) = Q_5-4P_0,\\
u_4  &=  \frac{(y+x^3)(y+x^3+x^2)^2}{y(y+x)(x^2+x+1)^3}, & \textrm{with } &(u_4) = Q_1+2Q_3-3Q_5,\\
u_5  &=  \frac{(y+x^3)^2(y+x^3+x^2+1)}{(y+1)(y+x)(x^2+x+1)^3},  & \textrm{with } & (u_{5}) = Q_4+2Q_1-3Q_5.
\end{align*}
Then consider the ray class field $K_{D'}$ of $\IF_{2}(E)$ of conductor $D'=4P_{1}+4P_{2}+4P_{3}$ in which the places in $S$ split completely. We are interested in the ray class fields $K_{D_{j}}$, $j=1,2,3$, that are subfields of $K_{D'}$ of conductor $D_{1}=4P_{1}+2P_{2}+2P_{3}$, $D_{2}=2P_{1}+4P_{2}+2P_{3}$ and $D_{3}=2P_{1}+2P_{2}+4P_{3}$. The corresponding $S$-ray class groups modulo $D_{j}$ are quotients of the groups $R_{j}=\Big(\mathbb{F}_2[[t_j]]/(t_j^4)\Big)^*\oplus \Big(\mathbb{F}_2[[{t_{j'}}]]/(t_{j'}^2)\Big)^*\oplus \Big(\mathbb{F}_2[[{t_{j''}}]]/(t_{j''}^2)\Big)^*\simeq \mathbb{Z}_4 \times \mathbb{Z}_2 \times \mathbb{Z}_2\times \mathbb{Z}_2$ by the image of the $S$-unit group of $E$ \cite[Section 8]{schoof}. Here $t_{j}$, $t_{j'}$, $t_{j''}$ denote uniformizers of $P_{j}$, $P_{j'}$, $P_{j''}$ respectively, for $\{j,j',j''\}=\{1,2,3\}$. We show that the order of the $S$-ray class group modulo $D_{j}$ is $2$ for $j=2$, while for $j=1,3$ this group is trivial. In Table \ref{uiPj}, we display in the column marked by $R_{j}$, $j=1,2,3$, the images of the $u_{i}$'s ($i\neq 3$) in the group $R_{j}$. We remark that the computations for the units $u_{4}$ and $u_{5}$ can be performed calculating the local expansions $y_j$ of $y$ at $P_j$, for $j=1,2,3$:
\begin{eqnarray*}
y_{1}&=&x+x^2+x^3+x^4+x^6+O(x^7),\\
y_{2}&=&1+x+x^2+x^3+x^4+x^6+O(x^7),\\
y_{3}&=&t^2+t^3+t^4+t^6+O(t^7),\, \textrm{where} \, t=x+1.
\end{eqnarray*}

\begin{table}[h]
\begin{center}
\begin{tabular}{| c | c | c | c |}
\hline
$u_{i}$ & $R_{1}$ & $R_{2}$ & $R_{3}$ \\
\hline
$u_{1}$ & $(1+t_1)^{3}(1+t_2)$ &$(1+t_1)(1+t_2)^3$ & $(1 + t_1) (1 + t_3^3)(1+t_2)$\\
\hline
$u_{2}$ & $(1+t_1^3)(1+t_3)$ & $(1+t_2^3)(1+t_3)$ & $(1+t_3)^3 $ \\
\hline
$u_{4}$ & $(1+t_1)^2(1+t_1^3)(1+t_2)$ & $(1+t_2)^3(1+t_2^3)$ & $1+t_2$\\
\hline
$u_{5}$ & $1+t_3$ & $(1+t_2^3)(1+t_3)$ & $(1+t_3)^3(1+t_3^3)$ \\
\hline
\end{tabular}
\caption{Images of the $u_{i}$'s in the group $R_{j}$, for $j=1,2,3$.}
\label{uiPj}
\end{center}
\end{table}

One checks that in $R_{2}$ the images of the $u_i$'s for $i\neq 3$ generate a subgroup of index $2$. The image of $u_{3}$ is $(1+t_1)(1+t_2)^3(1+t_2^3)(1+t_3)$ and lies hence in the same subgroup. On the other hand, the images of the $u_i$'s, $i\neq 3$, are independent generators of $R_{1}$: the image of $u_{1}$ has order $4$ and the images of $u_2$, $u_4$ and $u_5$ have order $2$. Thus in this case the ray class group is trivial. Similarly for the images of $u_1$, $u_2$, $u_4$ and $u_5$ in $R_{3}$: also in this case the ray class group is trivial.
\end{proof}

\begin{proposition}\label{possibleX}
All rational points of $E$ are ramified in $X'$. The curve $X'$ has genus $6$ and real Weil polynomial $h(t)=t(t+2)(t^2-5)^2$. In other words, only the configuration of case $I$ in Table \ref{typesIII} is possible.
\end{proposition}

\begin{proof}
According to Table \ref{typesIII}, there are three possibilities for the splitting behavior of the rational points of $E$ in $X$. Moreover by Lemmas \ref{genusx'} and \ref{XX'} the genus $g'$ and the vector $a(X')$ of the curve $X'$ are in the three cases as follows:
\begin{equation*}
\left.
\begin{tabular}{|l|c|c|c||c|c|}
 \hline
 $ \,$   & $\mathbf{a}$ & $\mathbf{b}$ & $\mathbf{c}$ & $\mathbf{g'}$ & $\mathbf{a(X')}$\\
\hline
\textbf{case } $I$ & $0$ & $5$ & $0$ & $6$ & $[5,0,0,10,\ldots]$\\
\hline
\textbf{case } $II$ & $1$ & $3$ & $1$ & $5$ &  $[5,1,0,10,\ldots]$\\
\hline
\textbf{case } $III$ & $2$ & $1$ & $2$ & $4$ & $[5,2,0,10,\ldots]$\\
\hline
\end{tabular}
\right.
\end{equation*}
Case $III$ cannot occur since in this case the curve $X'$ would be a genus $4$ curve having $N_4=N+2a_2+4a_4=5+2\cdot 2+4\cdot 10=49$ rational points over $\IF_{2^4}$, while $N_4(4)=45$ according to \cite{geervlugt}.\\ 
In case $II$ a computer calculation gives only one possible real Weil polynomial for $X'$, namely $h(t)=(t+2)(t^2-5)(t^2-2)$. Now, since the automorphism group of $E$ acts doubly transitively on $E(\IF_{2})$ as described in Remark \ref{remg1N5}, we may assume that the point at infinity $P_{0}$ of $E$ is the unique $A$-point of $E$ and that $P_{4}=(1,1)$ is the unique $C$-point. The remaining three rational points of $E$ are $\{P_{1}, P_{2}, P_{3}\}$. They are $B$-points of $E$ and hence ramify in $X' \to E$ by Lemma \ref{abcPoints}. Moreover, since $a_{4}(X')=10$ and $a_{4}(E)=5$, all five degree $4$ places of $E$ split completely in $X'$. By the Hurwitz formula the degree of the different of $\IF_{2}(X')/\IF_{2}(E)$ is $8$. Therefore Lemma \ref{rayclass2} implies that $\IF_{2}(X')$ is equal to the ray class field of $\IF_{2}(E)$ of conductor $2P_{1}+4P_{2}+2P_{3}$, in which $P_{0}$ and all degree $4$ places of $E$ split completely. Consider now the curve $\overline{X}$. It is a degree $3$ abelian covering of $X'$. Since $\overline{X}$ has $15$ rational points by Lemma \ref{XX'}, all five rational points of $X'$ split completely in $\overline{X}$. Moreover, since $X\to E$ and $X'\to E$ are both unramified outside of $E(\IF_{2})$, only the degree $2$ place $P_4'$ of $X'$, which lies over $P_4$ of $E$ ramifies in $\overline{X}$.
The curve $\overline{X}$ is hence the ray class field of $\IF_2(X')$ of conductor $P_4'$, where all rational places of $X'$ split completely. A computer calculation with MAGMA shows that the associated ray class group is trivial. Hence case $II$ cannot occur.\\
In case $I$ a computer calculation gives only one possible real Weil polynomial for $X'$, namely $h(t)=t(t+2)(t^2-5)^2$.
\end{proof}

In the next two lemmas we describe two curves appearing in the proof of Proposition \ref{Xb}.

\begin{lemma}\label{unhyper}
There exists a unique curve $C$ having real Weil polynomial $h(t)=(t+2)(t-1)$. Up to $\IF_{2}$-isomorphism, this is a genus $2$ projective curve described by the affine equation $y^2+xy=x^5+x^4+x^2+x$.
\end{lemma}

\begin{proof}
A curve $C$ having real Weil polynomial $h(t)=(t+2)(t-1)$ is a genus $2$ curve having four rational points and two places of degree $2$ over $\IF_{2}$. Since it is a hyperelliptic curve, we can consider the double covering $C\to \IP^{1}$. The different of the corresponding function field extension is $4P+2P'$, where $P$ and $P'$ are rational points of $\IP^1$. Indeed, by the Hurwitz formula, the degree of the different is $6$ and, since $C$ has four rational points, two of the rational points of $\IP^1$ are wildly ramified and one splits completely. The coefficients of $P$ and $P'$ are forced to be even since $\IF_{2}(C)$ is an Artin-Schreier extension of the rational function field. Notice also that the possibility that two rational points of $\IP^{1}$ split and the third stays inert in $\IF_{2}(C)$ is excluded by the fact that in this case the degree $2$ place of $\IP^{1}$ would be ramified, giving a contradiction in the computation of the different. According to the classification of genus $2$ curves over $\IF_{2}$ in \cite[page 327]{maisnernart}, by taking $P=P_\infty$ and $P'=(0,0)$, any such an hyperelliptic curve over $\IF_2$ is $\IF_2$-isomorphic to a projective curve of affine equation $y^2+y=x^3+ax+1/x+b$, with $a,b \in \IF_2$. There are hence four possibilities for the parameters $a$ and $b$, but only $y^2 + y = x^3 + x+1/x+1$ is the equation of a projective curve having four rational points over $\IF_{2}$ and two places of degree $2$. This curve is $\IF_2$-isomorphic to the projective curve of more simple affine equation $y^2+xy=x^5+x^4+x^2+x$, an isomorphism being given by $(x,y)\mapsto (x,(y+x^2)/x)$.
\end{proof}

\begin{lemma}\label{rayclass3}
Let $C$ be the curve of Lemma \ref{unhyper}. Then $C$ admits an unramified cyclic degree $5$ covering in which both the point at infinity $P_{\infty}$ and the point $(0,0)$ split. This covering is unique up to isomorphism. Moreover, for any other choice of rational points $P$ and $P'$ of $C$, any cyclic unramified degree $5$ covering of $C$ in which $P$ and $P'$ split, is necessarily trivial.
\end{lemma}

\begin{proof}
Consider the maximal unramified extension $L$ of the function field $K$ of $C$ where $P_\infty$ splits completely. By class field theory, the Galois group $\Gal(L/K)$ is isomorphic to the quotient of the class group $Pic(C)$ by the subgroup generated by the image in $Pic(C)$ of the Frobenius element $\mathrm{Frob} \,P_{\infty} \in \Gal(L/K)$ of $P_{\infty}$. Hence $\Gal(L/K)\simeq Pic^{0}(C)$. Let $h(t)$ be the real Weil polynomial of $C$ as in Lemma \ref{unhyper}. By \cite[Theorem 5.1.15 (c)]{stichtenoth} the class number $\# Pic^{0}(C)$ of $C$ equals $L(1)$, where $L(t)$ is the numerator of the Zeta function of $C$. Since $L(1) = h(q+1)$ by \eqref{relation}, one has $\# Pic^0(C)= h(3)=10$. Therefore there exists a unique unramified cyclic degree $5$ extension $K'$ of $\IF_{2}(C)$ in which $P_{\infty}$ splits completely. Since the divisor $(x)=2((0,0)- P_\infty)$ is principal, the Frobenius of $(0,0)$ is trivial in $\Gal(K'/K)\simeq\mathbb{Z}_5\simeq Pic^0(C)/\mathbb{Z}_{2}$, so that the rational point $(0,0)$ is also split in $K'$.\\
On the other hand, if we replace the points $P_{\infty}$ and $(0,0)$ by any other pair of rational points of $C$, there is no such unramified cyclic degree $5$ extension. To see this, we note that $C$ has four rational points: $P_{\infty}$, $(0,0)$, $(1,0)$ and $(1,1)$. If two of these were to split in an unramified cyclic degree $5$ covering of $C$, then $2$ times their difference, would be a principal divisor. By adding or subtracting the principal divisors $2((0,0)-P_{\infty})$ and $2((1,0)+(1,1)-2P_{\infty})$, this boils in each case down to the question of whether or not the divisor $2((1,0)-P_{\infty})$ is principal. Suppose that $2((1,0)-2P_{\infty})$ is the divisor of a function $f \in \IF_2(C)$. Since the only functions in $\IF_2(C)$ with a pole of order $2$ at infinity are linear functions in $x$, we must have $f=x+1$, but then $f$ also vanishes in $(1,1)$, a contradiction.
\end{proof}

\begin{proof}[Proof of Proposition \ref{Xb}]
By Lemma \ref{nonnormal6} the genus $6$ curve $X$ is a non-Galois covering of degree $3$ of the elliptic curve $E$. Moreover, by Proposition \ref{possibleX} the only possibility for the splitting behavior of the rational points of $E$ in $X$ is described in case $I$ of Table \ref{typesIII}. In other words, all rational points of $E$ are $B$-points in the sense of Definition \ref{abcpointsE}. In order to show that such a curve $X$ is unique, consider the quadratic function field extension $\IF_{2}(X')/\IF_{2}(E)$ described in the picture of Definition \ref{defioverlineprime}. By the Hurwitz formula and Proposition \ref{possibleX}, this is an abelian extension of $\IF_{2}(E)$ of conductor $\sum_{i=0}^4 2 P_i$ where all places of $E$ of degree $4$ split completely.\\
Let $\tau$ be the order $4$ automorphism of $E$ described in Remark \ref{remg1N5}. Then the endomorphism $\tau+2$ of the elliptic curve $E$ has degree $5$ and kernel $E(\IF_{2})$. The Galois group of the covering $\tau +2:\, E \to E$ consists of the translations by the points $P_{i}$ of $E$. It preserves both the set $E(\IF_{2})$ and the set of places of $E$ of degree $4$. Therefore the covering
\[
X' \longrightarrow E \stackrel{\tau +2}{-\!\!-\!\!\!\longrightarrow} E
\]
is Galois. Similarly, the covering $\overline{X}\to X'$ is unramified and cyclic of degree $3$. Lemma \ref{XX'} implies that all rational points of $X'$ are split. 
By class field theory, there exists a unique degree $3$ such a covering of $X'$. Indeed, let $h(t)$ be the real Weil polynomial of $X'$ as in Proposition \ref{possibleX}. By \cite[Theorem 5.1.15 (c)]{stichtenoth} one has $\# Pic^{0}(X')=L(1)$, where $L(t)$ is the numerator of the Zeta function of $X'$. Hence, since by \eqref{relation} one has $L(1)=h(3)=2^4\cdot 3\cdot 5$, there exists a unique index $3$ subgroup in the class group of $X'$. Thus the function field extension corresponding to the covering
\[
\overline{X} \longrightarrow E  \stackrel{\tau +2}{-\!\!-\!\!\!\longrightarrow} E
\]
is also Galois. The Galois group $G$ is an extension of $\mathbb{Z}_{5}$ by  $S_3$. Since these groups have coprime order and $\mathbb{Z}_{5}$ necessarily acts trivially on $S_{3}$, the Schur-Zassenhaus Theorem implies that $G$ is a direct product of $\mathbb{Z}_{5}$ and $S_{3}$. By Galois correspondence there exists hence a tower of function fields corresponding to the morphisms of curves $\overline{X}\to Y \to E$, such that $\Gal (\IF_{2}(Y)/\IF_{2}(E))\simeq S_3$. Let $\rho$ be a generator of $\Gal(\IF_2(\overline{X})/ \IF_2(X))\subseteq S_{3}$ and consider invariant fields. We obtain a cyclic covering $X \to C$ of degree $5$, which is unramified since $\tau+2: E\to E$ is.
\begin{center}
$
\xymatrix { 
Y \ar @{->}[d]_{\scriptscriptstyle 2}  \ar @{<-}[r]^{\scriptscriptstyle 5} & \overline{X} \ar @{->}[d]^{{\scriptscriptstyle 2}} \\
C \ar @{->}[d]_{\scriptscriptstyle 3}  \ar @{<-}[r]^{\scriptscriptstyle 5} & X \ar @{->}[d]^{{\scriptscriptstyle 3}} \\
E \ar @{<-}[r]^{\scriptscriptstyle 5} & E}
\quad\quad$
$
\xymatrix { 
\quad\;\mathbb{Z}_{5}\quad \ar @{->}[d]_{\scriptscriptstyle 2}  \ar @{<-}[r]^(.60){\scriptscriptstyle 5} & \{1\} \ar @{->}[d]^{{\scriptscriptstyle 2}} \\
\mathbb{Z}_{5}\times \mathbb{Z}_{2} \ar @{->}[d]_{\scriptscriptstyle 3}  \ar @{<-}[r]^(.60){\scriptscriptstyle 5} & \mathbb{Z}_{2} \ar @{->}[d]^{{\scriptscriptstyle 3}} \\
\mathbb{Z}_{5} \times S_{3} \ar @{<-}[r]^(.60){\scriptscriptstyle 5} & S_{3}}
$
\end{center}
The curve $C$ has genus $2$ by the Hurwitz formula. The real Weil polynomial of $C$ is thus a degree $2$ factor of the real Weil polynomial of $X$. Since $C$ is also a degree $3$ covering of $E$, the real Weil polynomial of $C$ is divisible by the real Weil polynomial $t+2$ of $E$, since the same holds for the corresponding Zeta functions \cite{aubryperret}. Hence the real Weil polynomial of $C$ is $h(t)=(t+2)(t-1)$. By Lemma \ref{rayclass3}, the curve $C$ indeed admits such an unramified cyclic degree $5$ covering. Therefore there actually exists a unique curve $X$ with real Weil polynomial equal to polynomial \eqref{rWb} in Theorem \ref{RealWeil} and Proposition \ref{Xb} follows.
\end{proof}

\section{Genus $7$ optimal curves}\label{sec: genus7}

Let $E$ be the optimal genus $1$ curve of affine equation $y^{2}+y=x^{3}+x$ described in Remark \ref{remg1N5}. In this last section we present a class field theoretic construction of a ray class field of $\IF_{2}(E)$ whose proper quadratic subfields are function fields of optimal genus $7$ curves. We show that the Zeta functions of these curves are not all the same, providing existence of at least two non-isomorphic genus $7$ optimal curves over $\IF_{2}$.

\begin{proposition}\label{g7N10}
Let $K$ be the function field of $E$ and let $Q$ denote a degree $6$ place of $K$ of uniformizer $t=x^6+x^5+1$. Let $L$ be the ray class field of $K$ of conductor $2Q$, in which all five rational points of $K$ split completely. The Galois group $\Gal(L/K)$ is isomorphic to $\mathbb{Z}_2 \oplus \mathbb{Z}_2$. The quadratic subfields of $K$ are function fields of optimal genus $7$ curves that do not all have the same Zeta function.
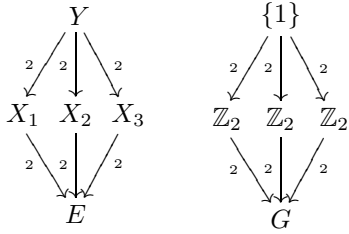
\begin{center}
$
\xymatrix@C=0.1pc { 
& {\;Y} \ar @{->}[dl]_(0.57){{\scriptscriptstyle 2}} \ar @{->}[d]_{{\scriptscriptstyle 2}} \ar @{->}[dr]^(0.57){{\scriptscriptstyle 2}} & \\
{\;\;X_1} \ar @{->}[dr]_(0.43){{\scriptscriptstyle 2}} & {X_2} \ar @{->}[d]_{{\scriptscriptstyle 2}} & {X_3} \ar @{->}[dl]^(0.43){{\scriptscriptstyle 2}}\\
& E &} 
\quad$
$
\xymatrix@C=0.1pc { 
& \{1\} \ar @{->}[dl]_(0.57){{\scriptscriptstyle 2}} \ar @{->}[d]_{{\scriptscriptstyle 2}} \ar @{->}[dr]^(0.57){{\scriptscriptstyle 2}} & \\
{\;\;\mathbb{Z}_2} \ar @{->}[dr]_(0.43){{\scriptscriptstyle 2}} & {\mathbb{Z}_2} \ar @{->}[d]_{{\scriptscriptstyle 2}} & {\mathbb{Z}_2} \ar @{->}[dl]^(0.43){{\scriptscriptstyle 2}}\\
& G &} 
\quad$
\end{center}
\end{proposition}

\begin{proof}
Let $a \in \IF_{2^{6}}$ be a root of $x^6+x^5+1$, and let $Q$ be the place that consists of the point $(a,a^4+a^3+a^2+1)$ and its conjugates. The prime ideal corresponding to $Q$ is $\mathfrak{p}=(x^6+x^5+1,y+x^4+x^3+x^2+1)$. The principal divisor $(x^6+x^5+1)$ is equal to $Q+Q'-12P_0$ where $Q'$ is the place consisting of $(a,a^4+a^3+a^2)$ and its conjugates. We take $t=x^6+x^5+1$ as a uniformizer at $Q$. Denote by $S$ the set of the five rational points of $E$ described in Remark \ref{remg1N5}.\\
Let $L$ be the ray class field of $K$ of conductor $2Q$, in which all five rational points in $S$ split completely. Then, by Artin reciprocity, the Galois group $G=\Gal(L/K)$ is isomorphic to the quotient of $R=\IF_{2^6}[[t]]^*/\{u:\, u\equiv 1\;\, \textrm{mod}\, t^2\}$ by the image of the $S$-unit group $O_S^*$ of $K$. A basis for $O_{S}^{*}$ is given by the functions $x$, $x+1$, $y$ and $y+x$ having the following principal divisors
\begin{eqnarray*}
(x) & = & P_1+P_2-2P_0,\\
(x+1) & = & P_3+P_4-2P_0,\\
(y) & = & P_1+2P_3-3P_0,\\
(x+y) & = & 2P_1+P_4-3P_0.
\end{eqnarray*}
In order to compute the image of the $S$-units in $R$, we first observe that the image of the $S$-unit $x$ has order $63$ modulo $t$ and hence it generates the $63$-part of $R$. Then we compute
\begin{align*}
x^{63}-1 & \equiv  (x+1)t  &\,& \textrm{mod}\, t^2,\\
(x+1)^{63}-1 & \equiv xt  &\,& \textrm{mod}\, t^2,\\
y^{63}-1 & \equiv (x^5+x^2)t &\,& \textrm{mod}\, t^2,\\
(y+x)^{63}-1 & \equiv (x^5+x^4+x^3+x^2)t &\,& \textrm{mod}\, t^2.
\end{align*}      
Thus $\Gal(L/K)$ is isomorphic to the quotient of $\IF_2[x]/(x^6+x^5+1)$ by the additive subgroup $H$ generated by $x+1$, $x$, $x^5+x^2$ and $x^5+x^4+x^3+x^2$. This is a quotient group of order $4$ where all elements have order $2$. Hence $\Gal(L/K)\simeq\mathbb{Z}_2 \oplus \mathbb{Z}_2$.\\
The three subgroups of order $2$ of $\Gal(L/K)$ correspond to three coverings $X_1$, $X_2$ and $X_3$ of $E$ as in the diagram. Each curve $X_{i}$ has ten rational points over $\IF_2$, since all five rational places of $E$ split completely. Since the non-trivial characters of $\Gal(L/K)$ have conductor $2Q$, the different of each quadratic extension $\IF_{2}(X_{i})/\IF_{2}(E)$ has degree $12$ and the three curves have genus $7$ by the Hurwitz formula. Since $N_2(7)=10$ by Theorem $5$ in \cite{serre1}, they are three genus $7$ optimal curves over $\IF_2$.\\
To show that the curves are not all isomorphic it suffices to consider the number of places of degree $d$ of each curve $X_{i}$ for $d\leq 4$. Since the rational points of $E$ are all split and $E$ has no places of degree $2$ or $3$, none of the three curves $X_i$ has places of degree $2$ or $3$ either. Therefore a curve $X_{i}$ can only have places of degree $4$ if some places of $E$ of degree $4$ split completely in $X_{i}$. By class field theory, a place $P$ of $E$ splits completely in $X_i$ if and only if the image of the uniformizer of $P$ is trivial in the quotient $R_{i}$ of $R$ which is the ray class group of the covering $X_i \to E$. Consider the index $2$ additive subgroups $H_1=H+\langle x^3 \rangle$, $H_2=H+\langle x^2 \rangle$ and $H_3=H+ \langle x^3+x^2 \rangle$ of $\IF_2[x]/(x^6+x^5+1)$. The ray class group $R_{i}$ associated to the curve $X_i$ is isomorphic to the quotient group of $\IF_2[x]/(x^6+x^5+1)$ by $H_{i}$ for $i=1,2,3$. We present the results of the computation in Table \ref{tabfour}.
\begin{table}[h]
\begin{center}
\begin{tabular}{| c | c | c | c |}
\hline
$\mathbf{Q_{j}}$ & $\mathbf{u_{j}(x,y)}$ & $\mathbf{g_{j}(x)}$ & $\mathbf{H_{i}}$\\
\hline 
$Q_{1}$ &  $y+x^3$ & $x^5+x$ & $H_{2}$\\
\hline 
$Q_{2}$ & $y+x^3+1$ & $x^4$ & $H_{1}$\\
\hline 
$Q_{3}$ & $y+x^3+x^2$ & $x^5+x^3+x$ & $H_{3}$\\
\hline 
$Q_{4}$ & $y+x^3+x^2+1$ & $x^4+x^2$ & $H_{3}$\\
\hline 
$Q_{5}$ & $x^2+x+1$ & $x^5+x^3+x^2$ & $H_{1}$\\
\hline
\end{tabular}
\caption{\mbox{Splitting behavior of the degree $4$ places of $E$ in each curve $X_{i}$.}}
\label{tabfour}
\end{center}
\end{table}
The first column lists for $j=1,\ldots,5$ the degree $4$ places $Q_{j}$ of $E$ as in Remark \ref{remg1N5}. In the second and third column we display the uniformizers $u_{j}(x,y)$'s of the $Q_{j}$'s and the images $g_{j}(x)$'s in $\IF_{2}/(x^{6}+x^{5}+1)$ of the $u_{j}(x,y)$'s. In other words we have $u_{j}(x,y)^{63}-1\equiv g_{j}(x)t \; \textrm{mod}\, t^2$. In the last column we write $H_{i}$ for $i=1,2,3$ whenever $g_{j}(x)$ belongs to $H_{i}$. The curve $X_1$ has four places of degree $4$, since both $Q_{2}$ and $Q_{5}$ split. Similarly, also $X_{3}$ has four places of degree $4$. On the other hand the curve $X_2$ has only two places of degree $4$, since only $Q_{1}$ splits. Hence the two curves $X_{1}$ and $X_{2}$ are not isomorphic.
\end{proof}

\begin{remark}\label{remX13}
Let $\sigma$ and $\tau$ be the automorphisms of $E$ described in Remark \ref{remg1N5}. Then, the action of $\sigma$ on the places of degree $6$ of $E$ listed in Remark \ref{remg1N5}, is given by 
\[
T_1 \mapsto T_9 \mapsto T_3 \mapsto T_4 \mapsto T_{10} \mapsto T_1.
\]
Since the elliptic involution $\tau^{2}$ switches $T_{9}$ and $T_{10}$ we have that $\sigma^3 \tau^{2}$ preserves $T_{10}$. In terms of adding points on the elliptic curve $E$ one has $\sigma^{3}\tau^{2}: (x,y) \mapsto (1,1) - (x,y)$. A short computation shows that $\sigma^{3}\tau^{2}$ switches the functions $x^3$ and $x^2+x^3$ modulo the subgroup $H$ of $\IF_2[x]/(x^6+x^5+1)$. Therefore the curves $X_1$ and $X_3$ are actually isomorphic.
\end{remark}

For completeness we compute the real Weil polynomials of the optimal genus $7$ curves.
\begin{proposition}
For $i=1,2,3$, the real Weil polynomial $h_{i}(t)$ and the vector $a(X_{i})$ of the curve $X_i$ are
\begin{align*}
&h_{1,3}(t)\!=\!(t\!+\!2)(t^6\!+\!5t^5\!+\!3t^4\!-\!15t^3\!-\!15t^2\!+\!9t\!+\!8), &a(X_{1,3})\!=\![10,0,0,4,2,5,18,\ldots],\\
&h_{2}(t)\!=\!(t\!+\!2)(t^2\!+\!3t\!+\!1)(t^4\!+\!2t^3\!-\!4t^2\!-\!5t\!+\!2),  &a(X_{2})\!=\![10,0,0,2,4,11,12,\ldots].
\end{align*}
\end{proposition}

\begin{proof}
By Remark \ref{remX13} the curves $X_{1}$ and $X_{3}$ are isomorphic, therefore they have the same real Weil polynomial. In the proof of Proposition \ref{g7N10} we already observed that for the curves $X_{1}$ and $X_{2}$ one has $a_{1}=10$ and $a_{2}=a_{3}=0$. We also proved that $a_{4}(X_{1})=4$ while $a_{4}(X_{2})=2$. Similarly to what was done for the places of degree $4$, we consider the splitting behavior of the places of degree $5$ of $E$ listed in Remark \ref{remg1N5} and display the results in Table \ref{tabfive}.

\begin{table}[h]
\begin{center}
\begin{tabular}{| c | c | c | c |}
\hline
$\mathbf{R_{k}}$ & $\mathbf{u_{k}(x,y)}$ & $\mathbf{g_{k}(x)}$ & $\mathbf{H_{i}}$\\ 
\hline 
$R_{1}$ &  $y+x^4$ & $x^3+x+1$ & $H_{1}$\\
\hline 
$R_{2}$ & $y+x^4+1$ & $x^5+x^4+x$ & $H_{3}$\\
\hline 
$R_{3}$ & $y+x^4+x$ & $x^4+x^3+x^2+1$ & $H_{2}$\\
\hline 
$R_{4}$ & $y+x^4+x+1$ & $x^5+x^4+x^3+1$ & $H_{2}$\\
\hline
\end{tabular}
\caption{\mbox{Splitting behavior of the degree $5$ places of $E$ in each curve $X_{i}$.}}
\label{tabfive}
\end{center}
\end{table}
Summing up we have $a_{5}(X_1)=2$ and $a(X_2)=[10,0,0,2,4,\ldots]$. Since the degree $6$ place $Q$ of $E$ is the only ramifying place in each curve $X_1$, $i=1,2$, we have that $a_6(X_i)$ has to be odd, while $a_7(X_i)$ has to be even. We can now determine a parametric form for the real Weil polynomial of each curve $X_i$: 
\begin{enumerate}
\item[$i)$] For the curve $X_1$ the values of $\#X_1(\IF_2)=a_1=10$, $a_2=a_3=0$, $a_4=4$ and $a_5=2$ allow to determine the following parametric form:
\[
h(t)  =  t^7+7t^6+13t^5-9t^4-45t^3-21t^2+\alpha t+\beta.
\]
One can check that only for the values of $(\alpha,\beta)=(26,16)$ and $(\alpha,\beta)=(27,18)$ all roots of $h(t)$ lie in the interval $[-2\sqrt{2},2\sqrt{2}]$. Only the first pair gives an odd number of degree $6$ places, namely $a_6(X_1)=5$. In this case $a_7(X_1)=18$. 
\item[$ii)$] For the values $a(X_2)=[10,0,0,2,4,\ldots]$ we have the parametric real Weil polynomial
\[
h(t)  =  t^7+7t^6+13t^5-9t^4-47t^3-33t^2+\alpha t+\beta.
\]
In this case there are three pairs of values of $(\alpha,\beta)$ for which $h(t)$ has all roots in the interval $[-2\sqrt{2},2\sqrt{2}]$: the pair $(3,2)$, which gives $a_6=10$; the pair $(4,4)$, which gives $a_6=11$; and the pair $(5,7)$, for which $a_6=12$. Hence the real Weil polynomial of $X_2$ corresponds to the unique pair $(\alpha,\beta)=(4,4)$ for which $a_6$ is not even. In this case $a_7=12$.
\end{enumerate}
\end{proof}

\noindent
\textsc{Alessandra Rigato}\\
K.U.~$\!$Leuven, Department of Mathematics,\\
Celestijnenlaan 200 B, B-3001 Leuven (Heverlee), Belgium\\
\verb+Alessandra.Rigato@wis.kuleuven.be+

\end{document}